\documentclass[titlepage,intlimits]{amsart}
\usepackage{latexsym,amsfonts,amssymb,amsmath,amsthm}
\newtheorem{theorem}{Theorem}[section]
\newtheorem{proposition}[theorem]{Proposition}
\newtheorem{lemma}[theorem]{Lemma}
\newtheorem*{TA}{Theorem A}
\newtheorem*{TB}{Theorem B}
\theoremstyle{remark}
\newtheorem{remark}[theorem]{Remark}
\numberwithin{equation}{section}
\DeclareMathOperator{\card}{card}
\DeclareMathOperator{\supp}{supp}
\newcommand{\JM}{Mierczy\'nski}
\newcommand{\Benaim}{Bena\"{\i}m}
\newcommand{\Mane}{Ma\~n\'e}
\newcommand{\Terescak}{Tere\v{s}\v{c}\'ak}
\newcommand{\JDE}{J. Differential Equations}
\newcommand{\pw}{a.e.}
\newcommand{\norm}[1]{\ensuremath{\lVert#1\rVert}}
\newcommand{\reals}{\ensuremath{\mathbb{R}}}
\newcommand{\Rn}{\ensuremath{\reals^{n}}}
\newcommand{\sphere}{\ensuremath{\mathbb{S}}}
\newcommand{\0}{\ensuremath{\{0\}}}
\newcommand{\x}{\ensuremath{\{x\}}}
\newcommand{\xI}{\ensuremath{x^{[I]}}}
\newcommand{\Breg}{\ensuremath{B_{\mathrm{reg}}}}
\newcommand{\N}{\ensuremath{\{1,\dots,n\}}}
\newcommand{\Merg}{\mathbf{M}_{\mathrm{erg}}}

\newcommand{\Cplus}{\ensuremath{C\setminus\0}}
\newcommand{\leI}{\ensuremath{\mathrel{{\leq}_{I}}}}
\newcommand{\lI}{\ensuremath{\mathrel{{<}_{I}}}}
\newcommand{\llI}{\ensuremath{\mathrel{{\ll}_{I}}}}

\newcommand{\VI}{\ensuremath{V_{I}}}
\newcommand{\VIi}{\ensuremath{V_{I\setminus i}}}
\newcommand{\CalB}{\ensuremath{\mathcal{B}}}
\newcommand{\CalBi}{\ensuremath{\CalB^{(i)}}}
\newcommand{\CalC}{\ensuremath{\mathcal{C}}}
\newcommand{\CalS}{\ensuremath{\mathcal{S}}}
\newcommand{\dC}{\ensuremath{\partial C}}
\newcommand{\Co}{\ensuremath{C^{\circ}}}
\newcommand{\CI}{\ensuremath{C_{I}}}
\newcommand{\CIo}{\ensuremath{\CI^{\circ}}}
\newcommand{\dCI}{\ensuremath{\partial\CI}}
\newcommand{\dSigma}{\ensuremath{\partial\Sigma}}
\newcommand{\Sigmao}{\ensuremath{\Sigma^{\circ}}}
\newcommand{\SigmaI}{\ensuremath{\Sigma_{I}}}
\newcommand{\SigmaIo}{\ensuremath{\SigmaI^{\circ}}}
\newcommand{\dSigmaI}{\ensuremath{\partial\SigmaI}}
\newcommand{\lambdamin}{\ensuremath{\lambda_{\mathrm{min}}}} %
\begin{document}
\title{On smoothness of carrying simplices}
\author{Janusz Mierczy\'nski}
\address{Institute of Mathematics\\
Wroc{\l}aw University of Technology\\
Wybrze\.ze Wyspia\'nskiego 27\\
PL-50-370 Wroc{\l}aw\\
Poland}
\email{mierczyn@banach.im.pwr.wroc.pl}
\thanks{First published in \emph{Proc. Amer. Math. Soc.} \textbf{127}(2), February 1999, pp. 543--551, published by the American Mathematical Society. \copyright\ 2016 American Mathematical Society.
\\
Research supported by KBN grant 2 P03A 076 08 (1995--97)}
\date{}
\commby{Hal L. Smith}
\begin{abstract}
We consider dissipative strongly competitive systems
$\dot{x}_{i}=x_{i}f_{i}(x)$ of ordinary differential equations.  It is
known that for a wide class of such systems there exists an invariant
attracting hypersurface $\Sigma$, called the carrying simplex.  In this
note we give an amenable condition for $\Sigma$ to be a $C^{1}$
submanifold-with-corners.  We also provide conditions, based on a recent
work of M. \Benaim\ \textit{On invariant hypersurfaces of strongly
monotone maps}, \JDE\ \textbf{137} (1997), 302--319, guaranteeing that
$\Sigma$ is of class $C^{k+1}$.
\end{abstract}
\maketitle

\section{Introduction}
\label{Introduction}
We consider systems of ordinary differential equations (ODE's) of class
(at~least) $C^{1}$
\begin{equation*}
{\dot x}_{i}=x_{i}f_{i}(x) \tag{E}
\end{equation*}
on the nonnegative orthant $C:=\{x\in\Rn:x_{i}\ge0$ for $1\le i\le n\}$,
$n\ge3$.

We write $F_{i}(x)=x_{i}f_{i}(x)$,
$F=(F_{1},\dots,F_{n})$.  The symbol $DF= [{\partial
F}_{i}/{\partial x}_{j}]_{i,j=1}^{n}$ stands for the
derivative matrix of the vector field $F$.  The local flow
generated by (E) on $C$ will be denoted by
$\phi=\{\phi_{t}\}$.  A subset $B\subset C$ is {\em
invariant\/} [resp.~{\em forward~invariant\/}] if
${\phi}_{t}x\in B$ for all $(t,x)\in\reals\times B$
[resp.~for all $(t,x)\in[0,\infty)\times B$] for which
${\phi}_{t}x$ is defined.  For $x\in C$, $B\subset C$ the
symbols $\omega(x)$, $\alpha(x)$, $\omega(B)$, $\alpha(B)$
have their usual meanings (see e.g.~Hale \cite{Hale}).  A
point $x\in C$ is a {\em rest point\/} if $\phi_{t}x=x$ for
each $t\in\reals$ (alternatively, if $F(x)=0$).  An
invariant subset $B$ of a compact invariant set $S$ is
called an {\em attractor\/} (resp.~a {\em repeller\/}) {\em
relative to $S$\/} if there is a relative neighborhood $U$
of $B$ in $S$ such that $\omega(U)=B$ (resp.~$\alpha(U)=B$).
For an attractor $B$ relative to $S$, by the repeller {\em
complementary\/} to $B$ we understand the set $\{x\in
S:\omega(x)\cap B=\emptyset\}$.  The attractor {\em
complementary\/} to a repeller $R$ is defined in an
analogous way.

System (E) is {\em dissipative\/} if there is a compact set
$B\subset C$ such that for each bounded $D\subset C$ its
$\omega$\nobreakdash-\hspace{0pt}limit set $\omega(D)$ is a
nonempty subset of $B$.  By standard results on global
attractors (see~\cite{Hale}), for a dissipative system
(E) there exists a compact invariant set $\Gamma\subset C$
(the {\em global attractor\/} for (E)) such that
$\omega(D)\subset\Gamma$ for each bounded $D\subset C$.
Evidently, $0\in\Gamma$.

For $I\subset\N$ denote
\begin{align*}
\CI&:=\{x\in C:x_{i}=0\text{ for }i\in I\},\\
\CIo&:=
\{x\in\CI:x_{j}>0\text{ for }j\notin I\},\\
\dCI&:=\CI\setminus\CIo.
\end{align*}
From the form of (E) it follows readily that any $\CI$, as
well as $\dCI$ and $\CIo$, is invariant.  We denote by
(E)$_{I}$ the restriction of system (E) to $\CI$.  Instead
of $C_{\emptyset}^{\circ}$, ${\partial}C_{\emptyset}$, we
write $\Co$, $\dC$.  $I'$ means $\N\setminus I$.

If system (E) is dissipative, so are all of its subsystems
(E)$_{I}$.  For each $I\subset\N$, the global attractor
$\Gamma_{I}$ for (E)$_{I}$ equals $\Gamma\cap\CI$.

System (E) is called {\em strongly competitive\/} if
$(\partial f_{i}/\partial x_{j})(x)<0$ for each $1\le
i,j\le n$, $i\ne j$, $x\in C$.  A strongly competitive
system is called {\em totally competitive\/} if $(\partial
f_{i}/\partial x_{i})(x)<0$ for $1\le i\le n$, $x\in C$.
Such systems describe a community of $n$ interacting
species where the growth of each species inhibits the
growth of any other.

Throughout the rest of the paper the standing assumption
will be:
\par
{\em \textup{(E)} is a $C^{1}$ dissipative strongly
competitive system of ODE's satisfying the following:
\par \textup{1.} $\0$ is a repeller relative to $\Gamma$.
\par \textup{2.} At each rest point $x\in C\setminus\0$ one
has $(\partial f_{i}/\partial x_{i})(x)<0$ for $1\le i\le
n$.}

The following important result was established by M. W.
Hirsch (\cite{Moe}).
\begin{proposition}
\label{projection}
The attractor $\Sigma\subset\Gamma$ complementary to the
repeller $\0$ is homeomorphic via radial projection to the
standard $(n-1)$\nobreakdash-\hspace{0pt}simplex
$\Delta:=\{x\in C:x_{1}+\dots+x_{n}=1\}$. Moreover, the
global attractor $\Gamma$ equals the convex hull of
$\Sigma\cup\0$.
\end{proposition}
Following M. L. Zeeman~\cite{Marylou}, the invariant
compact set $\Sigma$ is referred to as the {\em carrying
simplex\/} for (E).   In the ecological interpretation, the
carrying simplex can be thought of as expressing the
balance between the growth of small populations ($\0$ is a
repeller) and the competition of large populations
(dissipativity).

M. W. Hirsch in~\cite{Moe} asked about sufficient
conditions for the carrying simplex $\Sigma$ to be of class
$C^{1}$.  The time reverse flow $\{\phi_{-t}\}_{t\ge0}$
restricted to the invariant set $\Co$ is strongly monotone
and its derivative flow is strongly positive (for these
terms see H. L. Smith's monograph~\cite{Hal}).  Therefore,
when (E) possesses a repeller $R\subset\Sigma\cap\Co$
relative to $\Sigma$ we can utilize a powerful recent
result of I.~\Terescak~\cite{Ignac} on nonmonotone
manifolds to conclude that the repulsion basin
$B(R):=\{x\in\Sigmao:\alpha(x)\subset R\}$ is a $C^{1}$
hypersurface.  However, even in that case \Terescak's
theorem does not apply to the whole of $\Sigma$, for the
time reverse flow fails to be strongly monotone on the
boundary $\dC$.  Moreover, if we assume that (E) is
permanent (a natural assumption from the applied viewpoint)
then there is an attractor $A$ having the whole $\Co$ as
its attraction basin, hence its repulsion basin (relative
to $\Sigma$) equals $A$.  In his paper~\cite{JM} the
present author gave a fairly weak condition implying the
$C^{1}$ smoothness of $\Sigma$.  It was done, however, at
the expense of making use (for $n\ge5$) of Pesin's theory
of invariant measurable families of embedded manifolds,
which compels one to assume that $f$ has H\"older
continuous derivatives.

In this note we show that a well-known, robust, and readily
testable condition (see (A)) is enough to conclude that
$\Sigma$ is $C^{1}$.  Because our proofs exploit Oseledets'
theory of Lyapunov exponents, it suffices to assume $f$ to
be $C^{1}$ to get $C^{1}$ smoothness of $\Sigma$.  Next,
conditions are given, based on recent results of
M.~\Benaim~\cite{Michel}, for the carrying simplex to
possess higher~order smoothness.

I would like to thank Michel \Benaim\ for sending me a
preprint of~\cite{Michel}.

\section{Statement of main results}
\label{Statement}
For $I\subset\N$ put
\begin{equation*}
\SigmaI:=\CI\cap\Sigma, \quad
\SigmaIo:=\CIo\cap\Sigma,\quad
\dSigmaI:=\dCI\cap\Sigma.
\end{equation*}
We will call $\SigmaI$ a {\em
$k$\nobreakdash-\hspace{0pt}dimensional face\/} of
$\Sigma$, where $k=n-1-\card{I}$.  Evidently all $\SigmaI$,
as well as $\SigmaIo$ and $\dSigmaI$, are invariant. For
$I\subset\N$, the face $\SigmaI$ is the carrying simplex
for subsystem (E)$_{I}$. The
$0$\nobreakdash-\hspace{0pt}dimensional face $\Sigma_{i'}$
consists of a single rest point
$x^{(i)}=(0,\dots,0,x_{i}^{(i)},0,\dots,0)$ with
$x_{i}^{(i)}>0$ (called the {\em
$i$\nobreakdash-\hspace{0pt}th axial rest point\/}).

Let $V=\{v=(v_{1},\dots,v_{n}):v_{i}\in\reals\}$ stand for
the vector space of all free
$n$\nobreakdash-\hspace{0pt}dimensional vectors
(in~particular, we write the tangent bundle of the orthant
$C$ as $TC=C\times V$).  Depending on the context,
$\norm{\cdot}$ may mean the Euclidean norm of a vector, or
the operator norm of a matrix, associated with the
Euclidean norm.  For $I\subset\N$, we denote
\begin{equation*}
\VI:=\{v\in V:v_{i}=0\text{ for }i\in I\}.
\end{equation*}

For any two points $x$, $y\in\CI$, we write $x\leI y$ if
$x_{i}\le y_{i}$ for all $i\in I'$, and $x\lI y$ if $x\leI y$ and
$x\ne y$.  Moreover, $x\llI y$ if $x_{i}<y_{i}$ for all
$i\in I'$.  For $I=\emptyset$ we write simply $\le$, $<$,
$\ll$.  The reversed symbols are used in the obvious way.
As each $(\CI,\leI)$ is a lattice, we can define, for
$I\subset\N$ with $\card{I}\le n-1$
\begin{equation*}
\xI:=\bigvee_{i\in I'}x^{(i)}.
\end{equation*} where
It is easy to see that $x^{[J]}\lI\xI$ for $I\subsetneq J$.

The following result probably belongs to the folk~lore in
the theory of competitive systems, but I have not been able
to locate its proof.
\begin{lemma}
\label{unorder}
For each $I\subset\N$ with $1\le\card{I}\le n-2$ we have
$y\lI\xI$ for all $y\in\SigmaI$.
\end{lemma}
\begin{proof}
Suppose to the contrary that there is $y\in\SigmaI$ not in
the $\lI$ relation to $\xI$.  Assume first that $y=\xI$,
that is, $\xI\in\SigmaI$.  For $i\in I'$, $j\in I'$,
$i\ne j$, we have $x^{[I]}_{j}>x^{(i)}_{j}=0$.  As
$f_{i}(x^{(i)})=0$, it follows by strong competitiveness
that $f_{i}(\xI)<0$ for $i\in I'$.  We have therefore
$F_{i}(\xI)=x^{[I]}_{i}f_{i}(\xI)<0$ for all $i\in I'$.
Consequently, $\phi_{t}\xI\llI\xI$ for $t>0$ sufficiently
small.  But $\SigmaI$ is invariant, so
$\phi_{t}\xI\in\SigmaI$ for all $t>0$.  We have thus
obtained two points in $\SigmaI$ related by $\llI$, which
contradicts Lemma~2.5 in Hirsch~\cite{Moe}. Assume that
$y\in\SigmaI$ is not in the $\leI$ relation to $\xI$.  Take
an index $k$ for which $y_{k}>x^{[I]}_{k}$.  Let $J\subset\N$
stand for the set of those indices $j$ for which $y_{j}=0$.
Evidently $k\in J'$ and
$I\subset J$. We have $y\in\SigmaI\cap
C_{J}^{\circ}=\Sigma\cap\CI\cap C_{J}^{\circ}=\Sigma\cap
C_{J}^{\circ}= \Sigma_{J}^{\circ}$.  As a consequence,
$y_{j}>x^{(k)}_{j}=0$ for $j\in J'$, $j\ne k$, and
$y_{k}>x^{[I]}_{k}=x^{(k)}_{k}$ (since $k\notin I$).  But
this means that $y\gg_{J}x^{(k)}$.  As both these points
are in $\Sigma_{J}$, this again is in contradiction to
Lemma~2.5 in~\cite{Moe}.
\end{proof}

We say (E) satisfies hypothesis (A) if
\par
{\em For each $1\le i\le n$ one has $f_{i}(x^{[i]})\ge0$.}
\par
In the light of the strong competitiveness, (A) can be
equivalently formulated as:
\par
{\em For each $I\subset\N$ with $1\le\card{I}\le n-1$ one
has $f_{i}(\xI)\ge0$ for $i\in I$.\/}

Hypothesis (A) is well known in the literature on
mathematical ecology.  Consider the Lotka--Volterra
competitive system
\begin{equation}
\label{Lotka-Volterra}
{\dot x}_{i}=x_{i}(b_{i}-\sum_{j=1}^{n}a_{ij}x_{j})
\end{equation}
with $b_{i}>0$, $a_{ij}>0$.  For~(\ref{Lotka-Volterra}) the
$i$\nobreakdash-\hspace{0pt}th axial rest point is given by
$x^{(i)}_{i}=b_{i}/a_{ii}$.  It is easy to see that (A)
is now equivalent to
\begin{equation*}
b_{i}\ge\sum_{\substack{j=1\\ j\ne i}}^{n}
a_{ij}\frac{b_{j}}{a_{jj}}
\quad\text{for each }1\le i\le n.
\end{equation*}

We are now in a position to state our main result.
\begin{TA}
Assume that \textup{(E)} satisfies \textup{(A)}.  Then the
carrying simplex $\Sigma$ is a $C^{1}$
submanifold-with-corners, neatly embedded in $C$.
\end{TA}
For submanifolds-with-corners, their neat embeddings,
see~\cite{JM}.

We state now some consequences of hypothesis (A).  System
(E)$_{I}$ is called {\em permanent\/} if there is
$\epsilon>0$ such that
$\liminf_{t\to\infty}\rho(\phi_{t}x,\dCI)\ge\epsilon$ for
each $x\in\CIo$, where $\rho$ stands for the Euclidean
distance between a point and a set.

\begin{proposition}
\label{permanence}
If \textup{(A)} is satisfied then each of the subsystems
\textup{(E)$_{I}$} is permanent.
\end{proposition}
\begin{proof}
In order not to encumber our presentation with too many
subscripts, we prove the assertion for $I=\emptyset$, that
is, for system (E) only.  For each $i$, $1\le i\le n$, we
have as a result of strong competitiveness and
Lemma~\ref{unorder} that $f_{i}(x)>0$ for all
$x\in\Sigma_{i'}$.  Take now a neighborhood $U_{i}$ of
$\Sigma_{i'}$ in $C$ of the form
\begin{equation*}
U_{i}=\{(x_{1},\dots,x_{n}):0\le x_{i}<\epsilon_{i},
(x_{1},\dots,x_{i-1},x_{i+1},\dots,x_{n})\in\tilde{U}_{i}\},
\end{equation*}
where $\epsilon_{i}>0$ and a relative neighborhood
$\tilde{U}_{i}$ of $\Sigma_{i'}$ in $C_{i'}$ are so small that
$f_{i}(x)>0$ for all $x\in U_{i}$.  As $\Gamma$ is the global
attractor for (E) and $\Sigma$ is the attractor relative to
$\Gamma$ complementary to $\0$, there exists a forward
invariant neighborhood $U$ of $\Sigma$ in $C$ with the
property that $\phi_{t}x\in U$ for $x\in\Cplus$ and
sufficiently large $t$.  Also, $U$ can be taken so small
that all the sets $\{x\in U:x_{i}<\epsilon_{i}\}$ are
contained in $U_{i}$.  Now observe that for $t$ so large
that $\phi_{t}x$ belongs to $U$ one has
\begin{equation*}
\frac{d(\phi_{t}x)_{i}}{dt}=F_{i}(\phi_{t}x)=
x_{i}f_{i}(\phi_{t}x)>0
\end{equation*}
as long as $(\phi_{t}x)_{i}<\epsilon_{i}$.  From this it
readily follows that
$\liminf_{t\to\infty}\rho(\phi_{t}x,\Sigma_{i'})
\ge\epsilon_{i}$ for any $x\in\Co$.
\end{proof}

In view of results on attractors contained in Hale~\cite{Hale}
we have the following.
\begin{lemma}
\label{prop-of-perm}
Under the assumptions of Proposition~\ref{permanence}, for
each $I\subset\N$ the invariant compact set $\dSigmaI$ is a
repeller relative to $\SigmaI$.
\end{lemma}
For $I\subset\N$ denote by $A_{I}$ the attractor (relative
to $\SigmaI$) complementary to $\dSigmaI$.  As $A_{I}$ can
be viewed as the global attractor for the semiflow
$\{\phi_{t}\}_{t\ge0}$ restricted to the connected metric
space $\SigmaIo$, a result of Gobbino and Sardella
(Thm.~3.1 in~\cite{GS}) yields that $A_{I}$ is connected.

The ecological interpretation of the property described in
Proposition~\ref{permanence} is as follows.  In each
subcommunity none of the species goes extinct, and invasion
of a proper subcommunity by others causes the populations
of the previously present species to shrink due to the
larger amount of competition.

Before formulating sufficient conditions for $\Sigma$ to be
of class $C^{k+1}$ we need to introduce some notation (we
follow \Benaim's paper~\cite{Michel}).  For $x\in A_{I}$,
$I\subset\N$ with $\card{I}\le n-2$, we denote by
$\lambda(x)$ the largest eigenvalue of the symmetrization
of the matrix $(-DF^{I}(x))$, where $DF^{I}:=[\partial
F_{i}/\partial x_{j}]_{(i,j)\in I'\times I'}$.  Further,
$d(x)$ stands for the square root of
\begin{equation*}
\min_{\substack{i\ne j\\ i,j\notin I}}
\frac{{\partial}F_{i}}{{\partial}x_{j}}(x)
\frac{{\partial}F_{j}}{{\partial}x_{i}}(x).
\end{equation*}
Put $\lambda_{I}:=\sup\{\lambda(x):x\in A_{I}\}$
and $d_{I}:=\inf\{d(x):x\in A_{I}\}$.

We say that (E) satisfying (A) fulfills (C) if for each $I$
with $0\le\card{I}\le n-2$ any one of the conditions (C1)
or (C2) holds:
\par (C1) $k\sup\{\norm{DF^{I}(x)}:x\in A_{I}\}<2(k+1)d_{I}$.
\par (C2) $k\lambda_{I}<2(k+1)d_{I}$.

\begin{TB}
Assume that a $C^{k+1}$ system \textup{(E)} satisfies
\textup{(A)} and \textup{(C)}. Then the carrying simplex
$\Sigma$ is a $C^{k+1}$ submanifold-with-corners.
\end{TB}

\section{Proof of Theorem A}
\label{Ce-one}
Put $\sphere$ to be the
$(n-1)$\nobreakdash-\hspace{0pt}dimensional sphere $\{v\in
V:\norm{v}=1\}$.  For a vector subspace $W$ of $V$ and
$0\le k\le\dim{W}$, the symbol $\mathbb{G}_{k}W$ denotes
the compact metrizable space of all
$k$\nobreakdash-\hspace{0pt}dimensional vector subspaces of
$W$, endowed with the standard topology: for any two
$Z_{1}$, $Z_{2}\in\mathbb{G}_{k}W$, their distance is
defined as the Hausdorff distance between
$Z_{1}\cap\sphere$ and $Z_{2}\cap\sphere$.

The linearization of (E) generates on $TC$ a linear
skew-product (local) flow $(\phi_{t}x,D\phi_{t}(x)v)$,
where $D\phi_{t_{0}}(x)v_{0}$ is the value at time $t_{0}$
of the solution of the variational equation
$\dot\xi=DF(\phi_{t}x)\xi$ with initial condition
$\xi(0)=v_{0}$.

For a linear subset $\CalC$ of the product bundle $B\times W$,
where $B\subset\Sigma$ and $W$ is a vector subspace of $V$,
we will denote by $\CalC_{x}$ the set of all those $v\in W$
such that $(x,v)\in\CalC$ (in other words,
$\x\times\CalC_{x}$ is the fiber of $\CalC$ over $x$).  A
linear subset $\CalC$ of $B\times W$ is called {\em
invariant\/} if for each $(x,v)\in\CalC$ and each
$t\in\reals$ one has $(\phi_{t}x,D\phi_{t}(x)v)\in\CalC$.

Denote the set of all ergodic measures supported on a
compact invariant $B\subset\Sigma$ by $\Merg(B)$.  The
multiplicative ergodic theorem of Oseledets (see e.g.~\Mane\
\cite{Mane}) assures us that if $B\times W$ is an invariant
bundle then for each $m\in\Merg(B)$ there exist an
invariant $m$\nobreakdash-\hspace{0pt}measurable set
$\Breg\subset B$ (the set of {\em regular points\/}), a
collection $\CalC_{1},\dots,\CalC_{l}$ of invariant linear
subsets given by $m$\nobreakdash-\hspace{0pt}measurable
maps
$\Breg\ni x\mapsto(\CalC_{k})_{x}\in\mathbb{G}_{d_{k}}W$
(the {\em Oseledets decomposition\/}) and a collection
$\Lambda_{1}<\dots<\Lambda_{l}$ of reals ({\em Lyapunov
exponents\/}) such that
\begin{enumerate}
\item $W=\bigoplus_{k=1}^{l}(\CalC_{k})_{x}$ for
$x\in\Breg$,
\item
\begin{equation*}
\lim_{t\to\pm\infty}\frac{\log\norm{D\phi_{t}(x)v}}{t}=
\Lambda_{k}
\end{equation*}
for $1\le k\le l$, $x\in\Breg$ and $v\in(\CalC_{k})_{x}$.
\end{enumerate}

\begin{lemma}
\label{carrier}
For each $m\in\Merg(\Sigma)$ there is
$I=I(m)\subsetneq\N$ such that the support $\supp{m}$
of $m$ is contained in $A_{I}$.
\end{lemma}
\begin{proof}
By ergodicity of $m$ and invariance of all $\SigmaIo$,
there is precisely one $I\subset\N$ such that
$m(\SigmaIo)=1$ and $m(\dSigmaI)=0$.  Further, as
points from $\SigmaIo\setminus A_{I}$ are wandering
(relative to $\Sigma$), one has $m(\SigmaIo\setminus
A_{I})=0$.
\end{proof}
Fix $m\in\Merg(\SigmaI)$ with $m(\SigmaIo)=1$, and put
$\CalB:=\SigmaI\times\VI$, $\CalBi:=\SigmaI\times\VIi$,
$i\in I$.  Evidently, $\CalB$ is a subbundle of $\CalBi$ of
codimension one.  From the structure of system (E) it follows
that the bundles $\CalB$, $\CalBi$ are invariant.  Denote
by $\Lambda_{1}<\Lambda_{2}\dots<\Lambda_{l}$ the Lyapunov
exponents on $\CalB$ for the ergodic measure $m$ (we will
call them the {\em internal Lyapunov exponents\/} for
$m$).  Among the Lyapunov exponents on $\CalBi$ there is
one (denoted by $\lambda^{(i)}(m)$) corresponding to the
measurable linear set $\CalC^{(i)}_{k}\subset\CalBi$
such that $(\CalC^{(i)}_{k})_{x}\subsetneq\VI$ for
$m$\nobreakdash-\hspace{0pt}\pw\ $x\in\SigmaIo$.  We will
refer to $\lambda^{(i)}(m)$ as the
$i$\nobreakdash-\hspace{0pt}th {\em external Lyapunov
exponent\/} for $m$ (this terminology is modeled on
Hofbauer's~\cite{Hof}).

The following result was essentially proved in the author's
paper~\cite{JM} (except for terminology).
\begin{theorem}
\label{too-abstract-theorem}
Assume that for each $m\in\Merg(\dSigma)$ all its
external Lyapunov exponents are nonnegative.  Then the
following holds:
\begin{enumerate}
\item \label{t-a-1}
The carrying simplex $\Sigma$ is a $C^{1}$
submanifold-with-corners neatly embedded in $C$.
\item \label{t-a-2}
There are $\mu>0$ and an invariant one-dimensional
subbundle $\CalS$ of $\Sigma\times V$ such that for each
$m\in\Merg(\Sigma)$ one has
\begin{enumerate}
\item $\Sigma\times V=T\Sigma\oplus\CalS$, where $T\Sigma$
denotes the tangent bundle of $\Sigma$, and
$(\CalC_{1})_{x}\subset\CalS$ for
$m$\nobreakdash-\hspace{0pt}\pw\ $x\in\Sigma$.
\item $\Lambda_{1}$ is internal.
\item $\Lambda_{1}\le-\mu$.
\end{enumerate}
\end{enumerate}
\end{theorem}
In the present section we make use of part~1 of
Theorem~\ref{too-abstract-theorem} only.

In view of the above result, we need to prove only the
following.
\begin{proposition}
\label{external-nonnegative}
Under the assumptions of Theorem~A, for each
$m\in\Merg(\dSigma)$ all its external Lyapunov exponents
are nonnegative.
\end{proposition}
\begin{proof}
Fix a measure $m\in\Merg(\SigmaI)$ with
$m(\SigmaIo)=1$, and an index $i\in I$.  By
Lemma~\ref{carrier}, $\supp{m}\subset A_{I}$.  Take a
regular point $x\in\supp{m}$ and a vector
$v\in(\CalBi)_{x}\setminus\VI$ such that its
$i$\nobreakdash-\hspace{0pt}th coordinate $v_{i}$ is
positive.  As
$({\partial}F_{i}/{\partial}x_{j})(\phi_{t}x)=0$
for $j\ne i$, and
$({\partial}F_{i}/{\partial}x_{i})(\phi_{t}x)=
f_{i}(\phi_{t}x)$, it follows that the
$i$\nobreakdash-\hspace{0pt}th coordinate
$(D\phi_{t}(x)v)_{i}$ is the solution of the
(nonautonomous) scalar linear  ODE
$\dot\eta=f_{i}(\phi_{t}x)\eta$ with initial condition
$\eta(0)=v_{i}$.  By strong competitiveness and
Lemma~\ref{unorder}, $f_{i}$ is positive on the compact
invariant set $A_{I}\subset\SigmaI$, hence there is $M>0$
such that $f_{i}(\phi_{t}x)\ge M$ for all
$(t,x)\in\reals\times A_{I}$.  The standard theory of
differential inequalities yields
\begin{equation*}
\liminf_{t\to\infty}\frac
{\log(D\phi_{t}(x)v)_{i}}{t}\ge M.
\end{equation*}
$\norm{\cdot}$ is the Euclidean norm on $\Rn$, therefore for
all $t\in\reals$ we have
$\norm{D\phi_{t}(x)v}\ge(D\phi_{t}(x)v)_{i}$.  By
regularity of $x$ we derive
\begin{equation*}
\lim_{t\to\infty}\frac
{\log{\norm{D\phi_{t}(x)v}}}{t}=\lambda^{(i)}(m)\ge M>0.
\end{equation*}
\end{proof}

\section{Proof of Theorem B}
\label{higher-order}
We begin by stating a result being an adaptation of a
theorem of M.~\Benaim.
\begin{theorem}
\label{Benaim's-theorem}
Assume that a $C^{k+1}$, $k=1,\dots$, system \textup{(E)}
satisfies
\begin{enumerate}
\item For each $m\in\Merg(\dSigma)$ all external Lyapunov
exponents are nonnegative.
\item There is $\eta>0$ such that for each
$m\in\Merg(\Sigma)$ the inequality
\begin{equation}
\label{higher-order-condition}
\Lambda_{1}(m)-(k+1)\Lambda_{2}(m)<-\eta
\end{equation}
holds, where $\Lambda_{1}(m)$ and $\Lambda_{2}(m)$
denote respectively the smallest and the second smallest
Lyapunov exponents (on $\Sigma\times V$) for $m$.
\end{enumerate}
Then the carrying simplex $\Sigma$ is a $C^{k+1}$
submanifold-with-corners.
\end{theorem}
\begin{proof}[Indication of proof]
Theorem~\ref{too-abstract-theorem}.\ref{t-a-2} asserts
that the tangent bundle $TC$ restricted to $\Sigma$
invariantly decomposes as the Whitney sum
$T\Sigma\oplus\CalS$, and for each $m\in\Merg(\Sigma)$
the smallest Lyapunov exponent
$\Lambda_{1}(m)\le-\mu<0$ is the exponential growth
rate of a vector from $\CalS$ while any of the remaining
Lyapunov exponents is the exponential growth rate of a
vector tangent to $\Sigma$.
This, together with (\ref{higher-order-condition}), gives,
with the help of Prop.~3.3 in~\cite{Michel} (based on a result
of S. Schreiber~\cite{Sch}), that there are $c\ge1$,
$\alpha>0$ and $\beta>0$ such that
\begin{align*}
\norm{D\phi_{t}(x)v}&\le ce^{-{\alpha}t}\norm{v}
\quad\text{for $t\ge0$, $(x,v)\in\CalS$,}\\
\intertext{and}
\frac{\norm{D\phi_{t}(x)v}}
{\norm{D\phi_{t}(x)w}^{k+1}}
&\le ce^{-{\beta}t}
\quad\text{for $t\ge0$, $x\in\Sigma$, $v\in(\CalS)_{x}$,
$w\in T_{x}\Sigma\setminus\0$.}
\end{align*}
The rest of the proof consists in applying the $C^{k+1}$
section theorem of Hirsch, Pugh and Shub~\cite{HPS}, as in
the proof of Thm.~3.4 in~\cite{Michel}.
\end{proof}
As a consequence of the above theorem and
Proposition~\ref{external-nonnegative}, we will have
Theorem~B once we prove the following.
\begin{proposition}
Assume that a $C^{k+1}$ system \textup{(E)} satisfies
\textup{(A)} and \textup{(C)}.  Then there exists $\eta>0$
such that for each $m\in\Merg(\Sigma)$ the
inequality~\textup{(\ref{higher-order-condition})} holds.
\end{proposition}
\begin{proof}
Take $m\in\Merg(\Sigma)$, and let $I\subset\N$ be such
that $m(\SigmaIo)=1$.  From Lemma~\ref{carrier} we have
$\supp{m}\subset A_{I}$.  By results contained in
Sections~3 and~4 of~\cite{Michel} it follows that
under assumption~(C) there is
$\eta_{I}>0$ such that
\begin{equation*}
\Lambda^{*}_{1}(m)-(k+1)\Lambda^{*}_{2}(m)<-\eta_{I}
\end{equation*}
for all $m$ supported on $A_{I}$, where
$\Lambda^{*}_{1}(m)$ [resp.~$\Lambda^{*}_{2}(m)$]
stands for the smallest [resp.~second smallest] internal
Lyapunov exponent for $m$.
Theorem~\ref{too-abstract-theorem}.\ref{t-a-2} gives
$\Lambda^{*}_{1}(m)=\Lambda_{1}(m)$.
Denote by $\lambdamin$ the smallest external Lyapunov
exponent for $m$.  If $\lambdamin\ge\Lambda^{*}_{2}(m)$
then $\Lambda^{*}_{2}(m)=\Lambda_{2}(m)$ and
the inequality~(\ref{higher-order-condition}) is
satisfied with $\eta_{I}$.
Assume that $\lambdamin<\Lambda^{*}_{2}(m)$.  Applying
Theorem~\ref{too-abstract-theorem}.\ref{t-a-2} and
Proposition~\ref{external-nonnegative} we obtain
$\Lambda_{1}(m)\le-\mu<0\le\lambdamin=
\Lambda_{2}(m)$.  Consequently,
$\Lambda_{1}(m)\le-\mu<0\le
(k+1)\Lambda_{2}(m)$.
It suffices to put
\begin{equation*}
\eta:=
\min\{\mu,\eta_{I}:I\subset\N,\card{I}\le n-1\}.
\end{equation*}
\end{proof}

\section{Discussion}
\label{Discussion}
\begin{remark}
\label{May-Leonard}
In formulating our results, we preferred the assumptions to
be easily tractable rather than the weakest possible or
covering a wide range of applications.  In~fact, they can
be substantially weakened, as the following example shows.

A celebrated Lotka--Volterra system due to May and
Leonard~\cite{ML} has the form
\begin{equation}
\label{May-L}
\begin{aligned}
\dot{x}_{1}&=x_{1}(1-x_{1}-{\alpha}x_{2}-{\beta}x_{3})\\
\dot{x}_{2}&=x_{2}(1-{\beta}x_{1}-x_{2}-{\alpha}x_{3})\\
\dot{x}_{3}&=x_{3}(1-{\alpha}x_{1}-{\beta}x_{2}-x_{3})
\end{aligned}
\end{equation}
with $0<\beta<1<\alpha$ and $\alpha+\beta>2$.  It is easily
verified that \eqref{May-L} is dissipative, totally
competitive and has five rest points on $C$: $0$
(repelling), $y$ with $y_{i}=1/(1+\alpha+\beta)$ and three
axial ones $x^{(i)}$ with $x_{i}^{(i)}=1$.  Furthermore,
$\dSigma$ is an attractor relative to $\Sigma$ with $\{y\}$
as its complementary repeller (see pp.~67--68 in the
book~\cite{HS} by Hofbauer and Sigmund).
As a consequence,
$\Merg(\Sigma)=
\{\delta_{y},\delta_{x^{(1)}},\delta_{x^{(2)}},
\delta_{x^{(3)}}\}$.  The Lyapunov exponents for the Dirac
delta on a rest point $x$ are simply the real parts of the
eigenvalues of $DF(x)$.  At $y$ the smallest Lyapunov
exponent is negative and the remaining ones are positive
(see~\cite{HS}).  A simple calculation shows that at an
axial rest point the unique internal Lyapunov exponent
equals $-1$ and the external exponents are $1-\beta>0$ and
$1-\alpha<0$.

Assume now $\alpha<2$.  Then for each ergodic measure on
$\Sigma$ the smallest Lyapunov exponent $-1$ is internal.
Applying ideas of the author's earlier paper~\cite{JM} we
prove that the carrying simplex $\Sigma$ for \eqref{May-L}
is a $C^{1}$ submanifold-with-corners.  Observe that if
$1<\alpha<1+1/l$ for some $l=2,\dots$, then
$-1-l(1-\alpha)$ is negative and one can deduce along the
lines of the proof of Theorem~\ref{Benaim's-theorem} to
conclude that $\Sigma$ is of class $C^{l}$.
\end{remark}
\begin{remark}
The systems (E) satisfying (A) [resp.~(A) and~(C)] are
robust in the sense that if we perturb $f$ in a
neighborhood of $\Sigma$ in the $C^{1}$ [resp.~in the
$C^{k+1}$] topology then the perturbed system possesses a
carrying simplex of class $C^{1}$ [resp.~$C^{k+1}$] (and
each of its subsystems (E)$_{I}$ is permanent).  This
can be proved by reasoning similar to that in the proof of
Cor.~4.3 in~\cite{Michel}.
\end{remark}
\begin{remark}
In~principle, results contained in Section~\ref{Ce-one}
should carry over to the case that we allow $f$ to depend
periodically on $t$, although finding an analog of (A)
might be tricky (for time-periodic Lotka--Volterra strongly
competitive systems, compare e.g.~\cite{Tineo}).
\end{remark}

\end{document}